\documentclass[11pt,reqno]{amsart}
\usepackage[margin=1.25in, footskip=.5in]{geometry}
\usepackage[T1]{fontenc}
\usepackage[american]{babel}
\usepackage{amsmath,amsfonts,amssymb,amsthm,bm,wasysym}
\usepackage{mathrsfs}
\usepackage{graphicx}
\usepackage{enumitem}
\usepackage{multirow}
\usepackage{caption,float}
\usepackage{subcaption}
\usepackage{tikz}
\usetikzlibrary{matrix,arrows,backgrounds,calc,chains,automata,positioning,patterns}
\usetikzlibrary{decorations,decorations.pathmorphing}
\usetikzlibrary{shapes.geometric,calc}

\usepackage[colorlinks,linkcolor=blue!50!black,citecolor=blue!50!black,pagebackref,hypertexnames=false, breaklinks]{hyperref}

\allowdisplaybreaks
\newtheorem{theorem}{Theorem}
\newtheorem{proposition}[theorem]{Proposition}
\newtheorem{corollary}[theorem]{Corollary}
\newtheorem{lemma}[theorem]{Lemma}
\theoremstyle{definition}

\newtheorem*{def*}{Definition}
\newtheorem{example}{Example}

\newtheorem{remark}{Remark}

\newtheorem*{thm*}{Theorem}
\newtheorem*{lem*}{Lemma}
\newtheorem*{prop*}{Proposition}
\newtheorem*{rem*}{Remark}

\def\id{\operatorname{id}}
\def\SG{{\rm SG}}
\numberwithin{equation}{section}
\title[Small gaps]{Minimal gap in the spectrum \\ of the Sierpinski gasket}
\date{\today}
\author{Patricia Alonso Ruiz}
\address{Department of Mathematics, Texas A{\&}M University, College Station, TX 3368-77843}
\email{paruiz@math.tamu.edu}
\thanks{Research partly supported by the NSF grant DMS~1951577.}
\begin{document}
\begin{abstract}
This paper studies the size of the minimal gap between any two consecutive eigenvalues in the Dirichlet and in the Neumann spectrum of the standard Laplace operator on the Sierpi\'nski gasket. The main result shows the remarkable fact that this minimal gap is achieved and coincides with the spectral gap. The Dirichlet case is more challenging and requires some key observations in the behavior of the dynamical system that describes the spectrum.
\end{abstract}
\maketitle

\section{Introduction}
The standard Laplacian on the Sierpi\'nski gasket (SG) is a non-negative self-adjoint operator with a pure point spectrum consisting of countable many non-negative eigenvalues with finite multiplicity and only accumulation point at infinity. The study of this spectrum goes back to works in physics by Rammal and Tolouse~\cite{RT82,Ram84}, where it was observed that suitable series of eigenvalues in the finite level approximations of SG produced an orbit of a particular dynamical system. This phenomenon was named in~\cite{FS92} \textit{spectral decimation}. There, Fukushima and Shima were able to describe the complete spectrum of the Laplacian on SG by tracing back the orbits of the quadratic polynomial describing the aforementioned dynamical system. 

\medskip

Having a complete description of the spectrum led to extensive research, see e.g.~\cite{KL93,Kig98,Tep98,Str05,Str12}, dealing with the study of its properties. Some of them are 
in stark contrast to more classical settings: The Weyl ratio does not have a limit and in particular the eigenvalue counting function $N(x)$ fails Weyl and Berry's conjecture $N(x)\sim c_d\mathscr{H}^d({\rm SG})x^{d/2}+o(x^{d/2})$ with $d$ the Hausdorff dimension and $\mathscr{H}^d$ the $d$-dimensional Hausdorff measure of ${\rm SG}$, c.f.~\cite{KL93}. The oscillations that prevent the existence of a Weyl ratio are partly due to the existence of large gaps in the spectrum, meaning that if $\lambda_{(1)}\leq \lambda_{(2)}\leq\ldots$ denote the eigenvalues of the Laplacian, then $\limsup_{n}\lambda_{(n+1)}/\lambda_{(n)}>1$, c.f.~\cite{GRS01,Str05,Zho10,HSTZ12}. 

\medskip

As it turns out, the existence of these gaps in the spectrum of $\SG$ is equivalent to the fact that the Julia set describing the spectrum is totally disconnected~\cite[Theorem 2]{HSTZ12}, a property shared by certain classes of fractals, see~\cite{Tep98,HSTZ12} and references therein. 
On the other hand, the presence of exponentially large gaps has some advantageous consequences: For instance, they are responsible for a ``better than usual'' convergence of the Fourier series analogue of a function in $L^2(\SG)$, see~\cite[Theorem 1]{Str05}. Also, they provide a natural candidate for dyadic intervals as described in Section~\ref{S:large_gaps}, whose properties may be useful in the 
study of estimates involving eigenfunctions.

\medskip

Most of the existent work has focused on the structure of large gaps~\cite{DSV99,GRS01,Str05,Zho10}, whereas smaller ones have eluded further investigation.
The study of small gaps in the spectrum of the Laplacian 
can become fairly challenging~\cite{DJ16,BBRR17} and 
the question addressed in the present paper investigates what happens on $\SG$: Is it possible to provide a \textit{uniform lower bound} for the small gaps in the spectrum of the Laplacian? 

\medskip

Theorem~\ref{T:distance_lower_bound} provides a positive and optimal answer: Any two consequent eigenvalues in the Dirichlet or in the Neumann spectrum of the Laplacian on SG are separated \textit{at least} by the spectral gap. In other words, the first and the second eigenvalues are the closest within the whole spectrum. 
Although the existence of \textit{a} lower bound might possibly be derived from an abstract argumentation, proving that the spectral gap \emph{is in fact the smallest} spacing between Dirichlet eigenvalues requires a couple of rather non-trivial observations. 
These constitute the core of Section~\ref{S:small_gaps} and have been established after a careful analysis of the inverse function describing the spectrum of the finite level approximations. Note also that the size of the spacings do \emph{not} appear in an straightforward increasing order: Numerical computations show for example that the third gap is actually smaller than the second, and the seventh gap is smaller than the third, see Table~\ref{F:induction_step_D_table}.

\medskip

This paper is structured as follows: Section~\ref{S:prel} reviews the construction of the Dirichlet and the Neumann spectrum of the Laplacian along with some facts about the large gaps that were not available in the literature in this form. The main result, Theorem~\ref{T:distance_lower_bound}, is presented and proved for both spectra in Section~\ref{S:small_gaps}. Especially the proof of the Dirichlet case relies on several key properties of the dynamical system and the iterative construction of the spectrum. Possible directions 
for future investigation 
are briefly outlined in Section~\ref{S:end}.

\section{Preliminaries and useful facts}\label{S:prel}
For each $m\geq 0$, let $V_m$ denote the vertex set of the finite $n$-level approximation of the Sierpinski gasket ($\SG$ in the sequel) as depicted in Figure~\ref{F:SG_approx}, and let $\Delta_m$ denote the associated graph Laplacian
\begin{equation}
\Delta_m u(p):=\sum_{q\stackrel{m}{\sim} p}(u(q)-u(p)),
\end{equation}
with either Dirichlet or Neumann boundary conditions. 
The standard Laplacian on $\SG$, denoted by $\Delta$ in throughout the paper, is expressible as a limit of suitably renormalized graph Laplacians
\[
\Delta u(x)=\lim_{m\to\infty}5^m\Delta_m u(x),
\]
see e.g~\cite{KL93}. 
Further details concerning the precise construction and properties can be found in the books~\cite{Kig01,Str06}. This section reviews how the Dirichlet and Neumann spectrum of $\Delta_m$ describe that of $\Delta$ and records several useful facts about its structure.
\begin{figure}[H]
\centering
\renewcommand{\arraystretch}{0.5}
\begin{tabular}{cccc}
\begin{tikzpicture}
\tikzstyle{every node}=[draw,circle,fill=black,minimum size=2pt, inner sep=0pt]
\draw ($(0:0)$) node () {} --++ ($(0:2)$) node () {} --++ ($(120:2)$) node () {} --++ ($(240:2)$) node () {};
\end{tikzpicture}
\hspace*{2em}&
\begin{tikzpicture}
\tikzstyle{every node}=[draw,circle,fill=black,minimum size=2pt, inner sep=0pt]
\draw ($(0:0)$) node () {} --++ ($(0:2/2)$) node () {} --++ ($(120:2/2)$) node () {} --++ ($(240:2/2)$) node () {};
\draw ($(0:2/2)$) node () {} --++ ($(0:2/2)$) node () {} --++ ($(120:2/2)$) node () {} --++ ($(240:2/2)$) node () {};
\draw ($(60:2/2)$) node () {} --++ ($(0:2/2)$) node () {} --++ ($(120:2/2)$) node () {} --++ ($(240:2/2)$) node () {};
\end{tikzpicture}
\hspace*{2em}&
\begin{tikzpicture}
\tikzstyle{every node}=[draw,circle,fill=black,minimum size=2pt, inner sep=0pt]
\draw ($(0:0)$) node () {} --++ ($(0:2/4)$) node () {} --++ ($(120:2/4)$) node () {} --++ ($(240:2/4)$) node () {};
\draw ($(0:2/4)$) node () {} --++ ($(0:2/4)$) node () {} --++ ($(120:2/4)$) node () {} --++ ($(240:2/4)$) node () {};
\draw ($(60:2/4)$) node () {} --++ ($(0:2/4)$) node () {} --++ ($(120:2/4)$) node () {} --++ ($(240:2/4)$) node () {};
\foreach \a in {0,60}{
\draw ($(\a:2/2)$) node () {} --++ ($(0:2/4)$) node () {} --++ ($(120:2/4)$) node () {} --++ ($(240:2/4)$) node () {};
\foreach \b in{0,60}{
\draw ($(\a:2/2)+(\b:2/4)$)node () {} --++ ($(0:2/4)$) node () {} --++ ($(120:2/4)$) node () {} --++ ($(240:2/4)$) node () {};
}
}
\end{tikzpicture}
\hspace*{2em}&
\begin{tikzpicture}
\tikzstyle{every node}=[draw,circle,fill=black,minimum size=2pt, inner sep=0pt]
\draw ($(0:0)$) node () {} --++ ($(0:2/8)$) node () {} --++ ($(120:2/8)$) node () {} --++ ($(240:2/8)$) node () {};
\draw ($(0:2/8)$) node () {} --++ ($(0:2/8)$) node () {} --++ ($(120:2/8)$) node () {} --++ ($(240:2/8)$) node () {};
\draw ($(60:2/8)$) node () {} --++ ($(0:2/8)$) node () {} --++ ($(120:2/8)$) node () {} --++ ($(240:2/8)$) node () {};
\foreach \a in {0,60}{
\draw ($(\a:2/4)$) node () {} --++ ($(0:2/8)$) node () {} --++ ($(120:2/8)$) node () {} --++ ($(240:2/8)$) node () {};
\foreach \b in{0,60}{
\draw ($(\a:2/4)+(\b:2/8)$)node () {} --++ ($(0:2/8)$) node () {} --++ ($(120:2/8)$) node () {} --++ ($(240:2/8)$) node () {};
}
}
\foreach \c in{0,60}{
\draw ($(\c:2/2)$) node () {} --++ ($(0:2/8)$) node () {} --++ ($(120:2/8)$) node () {} --++ ($(240:2/8)$) node () {};
\draw ($(\c:2/2)+(0:2/8)$) node () {} --++ ($(0:2/8)$) node () {} --++ ($(120:2/8)$) node () {} --++ ($(240:2/8)$) node () {};
\draw ($(\c:2/2)+(60:2/8)$) node () {} --++ ($(0:2/8)$) node () {} --++ ($(120:2/8)$) node () {} --++ ($(240:2/8)$) node () {};
\foreach \a in {0,60}{
\draw ($(\c:2/2)+(\a:2/4)$) node () {} --++ ($(0:2/8)$) node () {} --++ ($(120:2/8)$) node () {} --++ ($(240:2/8)$) node () {};
\foreach \b in{0,60}{
\draw ($(\c:2/2)+(\a:2/4)+(\b:2/8)$)node () {} --++ ($(0:2/8)$) node () {} --++ ($(120:2/8)$) node () {} --++ ($(240:2/8)$) node () {};
}
}
}


\end{tikzpicture}
\\ [1em]
$V_0$ \hspace*{2em}& $V_1$ \hspace*{2em}& $V_2$ \hspace*{2em}& $V_3$
\end{tabular}
\caption{Graph approximations of the Sierpinski gasket.}
\label{F:SG_approx}
\end{figure}
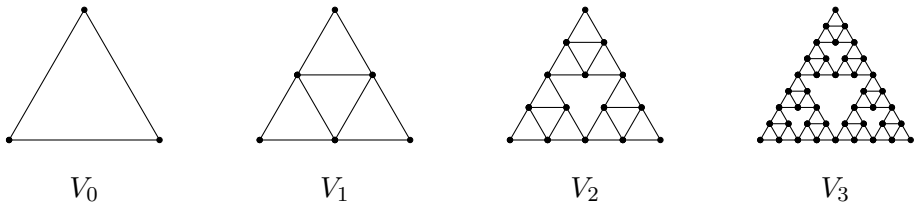
\subsection{Eigenvalues}
Fukushima and Shima described in~\cite{FS92} all Dirichlet eigenvalues of the Laplacian on $\SG$ by means of a genealogical tree as in Figure~\ref{F:genealogy}; a similar one can be made for the Neumann case. In this way, any eigenvalue $\lambda$ of $\Delta$ is related to a unique sequence of ancestors $\{\lambda_m\}_{m\geq j}$ starting at a particular generation $j\geq 1$, where $\lambda_m$ is an eigenvalue of $\Delta_m$. The starting generation $j$ is called the \emph{generation of birth} of $\lambda$ and first ancestors (who start a lineage) may only take the values $2,5$ or $6$. 
As reference to the first ancestor of $\lambda$, one speaks of an $i$-series eigenvalue, where $i=\lambda_j\in\{2,5,6\}$.

\begin{figure}[H]
\begin{tikzpicture}[auto]
\node (j) at ($(0:0)$) {$\lambda_j$};
\node[white] (j1m) at ($(225:1.5)$) {\small $\lambda_{j+1}$};
\node (j1p) at ($(315:1.5)$) {$\lambda_{j+1}$};
\node (j2m) at ($(315:1.5)+(225:1.5)$) {$\lambda_{j+2}$};
\node[white] (j2p) at ($(315:1.5)+(315:1.5)$) {\small $\lambda_{j+2}$};
\node (j3m) at  ($(315:1.5)+(225:1.5)+(225:1.5)$) {$\;\vdots$};
\node (j3p) at  ($(315:1.5)+(225:1.5)+(315:1.5)$) {};
\node (lim) at  ($(315:1.5)+(225:1.5)+(225:1.5)+(270:.75)$) {$\lambda$};
\draw[densely dotted,->] (j) to node [swap] {\footnotesize $\Phi_{-}$} (j1m);
\draw[very thick,->] (j) to node {\footnotesize $\Phi_{+}$} (j1p); 
\draw[densely dotted,->] (j1p) to node {\footnotesize $\Phi_{+}$} (j2p);
\draw[very thick, ->] (j1p) to node [swap] {\footnotesize $\Phi_{-}$} (j2m);
\draw[very thick, ->] (j2m) to node [swap] {\footnotesize $\Phi_{-}$} (j3m);
\draw[densely dotted, ->] (j2m) to node {\footnotesize $\Phi_{+}$} (j3p);
\end{tikzpicture}
\caption{Genealogy tree picture for an eigenvalue born at level $j$.}
\label{F:genealogy}
\end{figure}
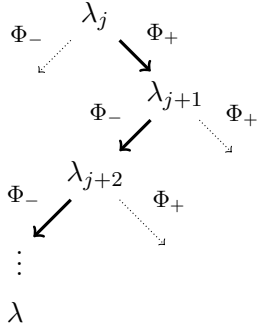

Fukushima-Shima's main result~\cite[Theorem 5.1]{FS92} states that all Dirichlet eigenvalues $\lambda$ satisfy
\begin{equation}\label{E:ev_as_lim_a}
\lambda=\lim_{m\to\infty}5^m\lambda_m
\end{equation}
and the same is true for the Neumann ones. As Rammal and Toulouse observed in~\cite{RT82}, the genealogy of an eigenvalue arises from the inverse functions of the quadratic polynomial $R(z)=z(5-z)$ given by
\begin{equation}\label{E:def_Phi_pm}
\Phi_{\pm}(z)=\frac{1}{2}\big(5\pm\sqrt{25-4z}\big),
\end{equation}
while latter relates each generation of an eigenvalue $\lambda$ to the immediate next by
\begin{equation}\label{E:ev_level_rel}
\lambda_{m}=\lambda_{m+1}(5-\lambda_{m+1}).
\end{equation} 
Once the first ancestor is born, its successor is determined by a choice of one of the inverse functions $\Phi_{-}$ or $\Phi_{+}$ and the same for all subsequent generations. The genealogy of an eigenvalue $\lambda$ with generation of birth $j\geq 1$ is thus described by a sequence $\{\lambda_{j+k}\}_{k\geq 0}$ with $\lambda_j:=i\in\{2,5,6\}$ and
\begin{equation}\label{E:lambda_as_phi}
\lambda_{j+k}
=\Phi_{w|_{k}}(\lambda_j)=\Phi_{w|_{k}}(i):=\Phi_{w_{k}}\circ \Phi_{w_{k-1}}{\circ}\cdots \circ \Phi_{w_1}(i)
\end{equation}
for any $k\geq 1$, where $w:=\ldots w_2w_1\in\{-,+\}^{\mathbb{N}}$. 
Such an eigenvalue is generically called an \emph{$i$-series}; in view of~\eqref{E:ev_as_lim_a} and~\eqref{E:lambda_as_phi} it also admits the expression
\begin{equation}\label{E:ev_as_lim_b}
\lambda
=5^{j}\lim_{k\to\infty}5^{k}\Phi_{w|_{k}}(i),
\end{equation}
which will play a fundamental role in the proof of Theorem~\ref{T:distance_lower_bound}. 
Due to the properties of $\Phi_{\pm}$, c.f. Section~\ref{S:small_gaps}, in order for the limit~\eqref{E:ev_as_lim_b} to exist there must be a specific generation $\ell\geq j$, called the \emph{generation of fixation} of $\lambda$, from which on all descendants are obtained via $\Phi_{-}$. In other words, 
for any $m\geq \ell$,
\begin{equation}\label{E:def_eof}
\lambda_m
=\Phi_{-}^{(m-\ell)}(\lambda_{\ell})=\Phi_{-}^{(m-\ell)}\Phi_{w|_{\ell-1}}(i),
\end{equation}
where $\Phi_{-}^{(n)}$ denotes the $n$-th concatenation of $\Phi_{-}$ and $\Phi_{w|_0}=\id$. Note that $w_{\ell-1}={+}$ as long as $\ell\geq 2$. 

\begin{example}\label{Ex:lowest_D}
The lowest $5$-series and $6$-series Dirichlet eigenvalue have generation of birth $1$ respectively  $2$, and generation of fixation $2$, respectively $4$. They admit the limit representation
\[
\lim_{k\to\infty}5^{k+1}\Phi_{w|_{k}}(5)=:\lambda_0^{(5)},\quad\text{and}\quad\lim_{k\to\infty}5^{k+2}\Phi_{w|_{k}}(6)=:\lambda_0^{(6)},
\]
where in the first case $w_{k}=-$ for all $k\geq 1$ and in the second $w_1=+$ and $w_{k}=-$ for all other $k\geq 2$. The notation used for these eigenvalues follows~\cite[Section 3]{DSV99} and depends on the word $w$ in a non-trivial way which we do not discuss here. In view of~\eqref{E:ev_as_lim_b}, the lowest $5$-series with generation of birth $j\geq 1$ satisfies
\[
\lambda=\lim_{k\to\infty}5^{j+k}\Phi_{w|_{k}}(5)=5^{j-1}\lim_{k\to\infty}5^{k+1}\Phi_{-}^{(k)}(5)=5^{j-1}\lambda_0^{(5)},
\]
and the lowest $6$-series with generation of birth $j\geq 2$ is
\[
\lambda=\lim_{k\to\infty}5^{j+k}\Phi_{w|_{k}}(6)=5^{j-2}\lim_{k\to\infty}5^{k+2}\Phi_{-}^{(k-1)}\Phi_{+}(6)=5^{j-2}\lambda_1^{(6)}.
\]
\end{example}

\subsection{Large gaps}\label{S:large_gaps}
The existence of large gaps in both the Dirichlet and the Neumann spectrum of the Laplacian has been extensively studied in the literature, see e.g.~\cite{Str05,Zho10,HSTZ12}. This section gives a brief account of that phenomenon in the case of the Dirichlet spectrum and records some observations of interest that had not appeared in this form yet.

\medskip

As a non-negative compact self-adjoint operator, $\Delta$ has a pure point spectrum, whose ordered elements are denoted as $0\leq \lambda_{(1)}\leq \lambda_{(2)}\leq\lambda_{(3)}\leq\ldots$ including multiplicity (the notation should not be confused with that of the eigenvalues of a finite level used in the previous section). Two of the most prominent recurrent gaps in the spectrum, described in~\cite[Theorem 1]{Str05} and~\cite[Theorem 5.1]{GRS01}, occur between the eigenvalues
\begin{equation}\label{E:order_largest_gaps}
\lambda_{(N_m-N_{m-1})}<\lambda_{(N_m)}<\lambda_{(N_m+1)},
\end{equation}
where $N_m=\frac{1}{2}(3^{m+1}-3)$, $m\geq 1$, coincides with the size of the spectrum of $\Delta_m$. The eigenvalues in~\eqref{E:order_largest_gaps} correspond to the lowest $5$-series, the lowest $6$-series and the second lowest $5$-series with generation of birth $j=m$ computed as in Example~\ref{Ex:lowest_D}. 
More precisely, 
\begin{equation}\label{E:largest_gaps}
\frac{\lambda_{(N_m)}}{\lambda_{(N_m-N_{m-1})}}
=\frac{\lambda_1^{(6)}}{5\lambda_0^{(5)}}\approx 2.425
\qquad\text{and}\qquad
\frac{\lambda_{(N_m+1)}}{\lambda_{(N_m)}}
=\frac{5\lambda_1^{(5)}}{\lambda_1^{(6)}}\approx 1.271,
\end{equation}
c.f.~\cite[Theorem 5.1]{GRS01}. 
There are more fractals for which similar statements hold, see e.g.~\cite{HSTZ12} and references therein.

\subsection{Dyadic intervals}
In applications, e.g.~\cite{Str05}, the largest gaps in~\eqref{E:largest_gaps} can be used to decompose the positive real line into the analogue of classical ``dyadic intervals''. For any $m\geq 2$, these may be defined as
\begin{equation}\label{E:def_Bk}
B_m:=\begin{cases}
[0,\lambda_{(N_2)}),&m=2,\\
[\lambda_{(N_{m-1})},\lambda_{(N_m)}),&m\geq 3,
\end{cases}
\end{equation}
see Figure~\ref{F:dyadic_blocks}. 

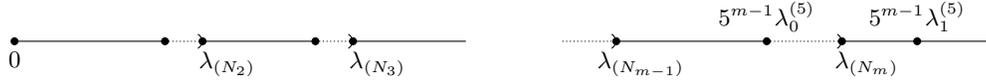
\begin{figure}[H]
\begin{tikzpicture}
\coordinate[label=below: {\footnotesize $0$}] (o) at ($(0:0)$);
\fill (o) circle (1.5pt);
\coordinate (l2L) at ($(0:2)$);
\fill (l2L) circle (1.5pt);
\coordinate[label=below: {\footnotesize $\qquad\lambda_{(N_2)}$}] (l2) at ($(0:2.5)$);
\fill (l2) circle (1.5pt);
\draw (o) -- (l2L);
\draw[densely dotted, -)] (l2L) -- (l2);
\coordinate (l3L) at ($(0:4)$);
\fill (l3L) circle (1.5pt);
\coordinate[label=below: {\footnotesize $\qquad\lambda_{(N_3)}$}] (l3) at ($(0:4.5)$);
\fill (l3) circle (1.5pt);
\draw (l2) -- (l3L);
\draw[densely dotted, -)] (l3L) -- (l3);
\draw (l3) --++ ($(0:1.5)$);
\coordinate[label=below: {\footnotesize $\qquad\lambda_{(N_{m-1})}$}] (lm_1) at ($(0:8)$);
\fill (lm_1) circle (1.5pt);
\draw[densely dotted, (-] (lm_1) --++ ($(180:.75)$);
\coordinate[label=above: {\footnotesize $5^{m-1}\lambda_0^{(5)}$}] (lmL) at ($(0:10)$);
\fill (lmL) circle (1.5pt);
\draw (lm_1) -- (lmL);
\coordinate[label=below: {\footnotesize $\qquad\lambda_{(N_m)}$}] (lm) at ($(0:11)$);
\fill (lm) circle (1.5pt);
\draw[densely dotted, -)] (lmL) -- (lm);
\coordinate[label=above: {\footnotesize $5^{m-1}\lambda_1^{(5)}$}] (lmR) at ($(0:12)$);
\fill (lmR) circle (1.5pt);
\draw (lm) --++ (0:2);
\end{tikzpicture}
\caption{Eigenvalues in dyadic intervals are separated by the largest gaps (dotted).}
\label{F:dyadic_blocks}
\end{figure}

As a consequence of~\eqref{E:largest_gaps}, the distance between any two intervals increases with their index; the distance is in fact comparable to the magnitude of the first eigenvalue in the larger one.

\begin{proposition}
Let $B_m,B_{m'}$ be intervals with $2\leq m< m'$. Then, 
\[
\min\{|\lambda-\lambda'|~\colon~\lambda\in B_m,\lambda'\in B_{m'}\}>\frac{1}{2}\lambda_{(N_{m'-1})}.
\]
\end{proposition}
\begin{proof}
Let $g_0:=\frac{\lambda_1^{(6)}}{5\lambda_0^{(5)}}>2$. By virtue of~\eqref{E:largest_gaps} and since $\lambda_{(N_{m'})}=5^{m'-m}\lambda_{(N_m)}$ by construction, we have
\[
\lambda_{(N_{m'-1})}-\lambda_{(N_{m}-N_{m-1})}
=\lambda_{(N_{m'-1})}(1-g_0^{-1}5^{m-1-m'})\geq \lambda_{(N_{m'-1})}(1-g_0^{-1})>\frac{1}{2}\lambda_{(N_{m'-1})},
\]
where the last inequality holds because $g_0>2$. 
\end{proof}
The last observation in this paragraph refers to the fact that the gaps are actually so large that the sum of any two eigenvalues, which may belong to the same or to different intervals, always remains within the larger interval.

\begin{proposition}
If $\lambda\in B_m$, $\lambda'\in B_{m'}$ are eigenvalues with $2\leq m\leq m'$, then $\lambda+\lambda'\in B_{m'}$.
\end{proposition}
\begin{proof}
Setting again $g_0:=\frac{\lambda_1^{(6)}}{5\lambda_0^{(5)}}>2$, it follows from~\eqref{E:def_Bk} and~\eqref{E:largest_gaps} that
\begin{multline*}
 \lambda_{(N_{m'-1})}<\lambda+\lambda'\leq \lambda_{(N_{m}-N_{m-1})}+\lambda_{(N_{m'}-N_{m'-1})}\\
 =g_0^{-1}(\lambda_{(N_{m})}+\lambda_{(N_m')})\leq 2g_0^{-1}\lambda_{(N_{m'})}<\lambda_{(N_{m'})}.
\end{multline*}
\end{proof}

\section{Small gaps}\label{S:small_gaps}
The main result of the present paper establishes another remarkable property of both the Dirichlet and the Neumann spectrum of $\Delta$: The minimal spacing between any two distinct eigenvalues \emph{equals} the corresponding spectral gap and hence provides an optimal lower bound for the size of small gaps in the spectrum. The description of the eigenvalues as limits in~\eqref{E:ev_as_lim_a} already suggests that the proof will rely in a careful analysis of the corresponding spectrum of the finite Laplacian $\Delta_m$.

\subsection{Minimal eigenvalue spacing}
To state the theorem precisely, let $\lambda_0^{(2)}$ and $\lambda_0^{(5)}$ denote the lowest $2$-series and $5$-series Dirichlet eigenvalue, and $\lambda_0^{(6)}$ the lowest non-zero Neumann eigenvalue of $\Delta$.  
As limits of the form~\eqref{E:ev_as_lim_b} they admit the expression
\begin{equation}\label{E:spectral_gaps_SG}
\lambda_0^{(2)}=\lim_{k\to\infty}5^{k+1}\Phi_{-}^{(k)}(2),\quad \lambda_0^{(5)}=\lim_{k\to\infty}5^{k+1}\Phi_{-}^{(k)}(5),\quad
\lambda_0^{(6)}=\lim_{k\to\infty}5^{k+2}\Phi_{-}^{(k-1)}\Phi_{+}(6)
\end{equation}
and also satisfy $\lambda_0^{(2)}<\lambda_0^{(5)}<\lambda_0^{(6)}$, c.f.~\cite[Theorem 5.1]{GRS01}, see also Example~\ref{Ex:lowest_D}. 

\begin{theorem}\label{T:distance_lower_bound}
The spacing between any two distinct eigenvalues in the (Dirichlet, or Neumann) spectrum of $\Delta$ is bounded below by the corresponding spectral gap. Precisely,
\begin{equation}\label{E:min_dist_Dir}
\min\{|\lambda-\lambda'|\colon \lambda\neq\lambda'\text{ \rm Dirichlet eigenvalues of }\Delta\}=\lambda_0^{(5)}-\lambda_0^{(2)}
\end{equation}
and
\begin{equation}\label{E:min_dist_Neum}
\min\{|\lambda-\lambda'|\colon \lambda\neq\lambda'\text{ \rm Neumann eigenvalues of }\Delta\}=\lambda_0^{(6)}.
\end{equation}
\end{theorem}

\begin{proof}[Proof of Theorem~\ref{T:distance_lower_bound}]
We prove the result in the Dirichlet case, the Neumann case follows similarly. Let $\lambda,\lambda'$ be two distinct Dirichlet eigenvalues with generations of birth $j,j'\geq 1$ respectively, generations of fixation $\ell\geq j,\ell'\geq j'$, and associated sequences $\{\lambda_{j+k}\}_{k\geq 0}$ and $\{\lambda_{j'+k}\}_{k\geq 0}$. Without loss of generality we may assume $\ell\leq \ell'$ and $\lambda_\ell\leq \lambda_{\ell'}$. 
By construction, c.f.~\eqref{E:def_eof}, for any $m\geq \ell'$ we have 
$\lambda_{m},\lambda_{m}'\in\Phi_{-}^{(m-\ell')}(A^D_{\ell'}{\setminus}\{6\})$, where $A^D_{\ell'}$ denotes the Dirichlet spectrum of $\Delta_{\ell'}$. 
By virtue of Lemma~\ref{L:pre-lowest_D},
\begin{multline*}
5^{m}\lambda_m-5^m\lambda_m'
\geq  5^{m}\min\{|\lambda-\lambda'|\colon \lambda,\lambda'\in \Phi_{-}^{(m-\ell')}(A^D_{\ell'}{\setminus}\{6\})\}\\
\geq 5^m\Phi_{-}^{(m-1)}(5)-5^m\Phi_{-}^{(m-1)}(2)
\end{multline*}
and in view of~\eqref{E:spectral_gaps_SG}, letting $m\to\infty$ yields~\eqref{E:min_dist_Dir}. For the Neumann case, replace Lemma~\ref{L:pre-lowest_D} by Lemma~\ref{L:gap_AmN} noticing that $\Phi_{+}(6)=3$, and use the fact that zero is a Neumann eigenvalue. 
\end{proof}

\subsection{Key properties of the inverse functions}
We start by recording several useful properties of the functions $\Phi_{\pm}$ that describe the spectrum of $\Delta$. Recall from~\eqref{E:def_Phi_pm} that these are given by
\begin{equation*}
\Phi_{\pm}(z)=\frac{1}{2}\big(5\pm\sqrt{25-4z}\big)
\end{equation*}
and correspond to the inverse functions of the polynomial $R(z)=z(5-z)$. Their graphs, displayed in Figure~\ref{F:inv_fcts_plot}, provide a fairly good insight of the following observations.

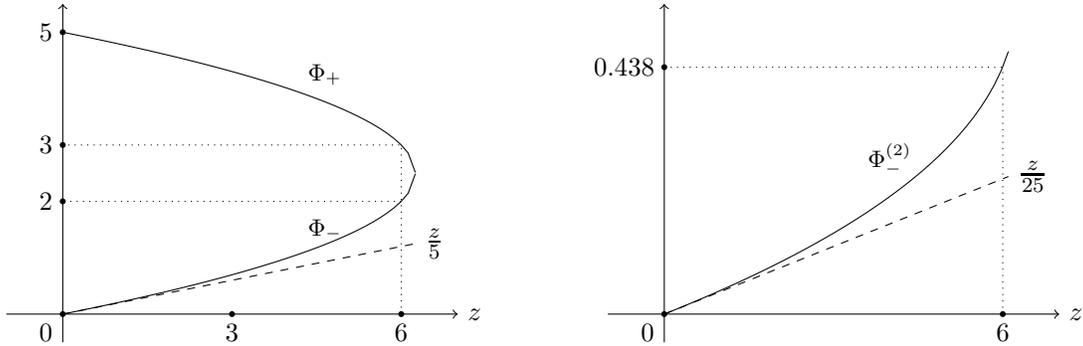
\begin{figure}[H]
\centering
\begin{tabular}{cc}
\begin{tikzpicture}[domain=0:25/4,samples=50,scale=.75]
\draw[->] ($(180:1)$) --++ ($(0:8)$) node[right] {\small $z$}; 
\draw[->] ($(270:.5)$) --++ ($(90:6)$) node {};
\draw plot (\x,{0.5*(5-sqrt(25-4*\x))}) node[right] {};
\draw plot (\x,{0.5*(5+sqrt(25-4*\x))}) node[right] {};
\draw[dashed] plot (\x,{\x/5}) node[right] {$\frac{z}{5}$};
\node (pm) at ($(0:4.65)+(90:1.5)$) {\footnotesize $\Phi_{-}$};
\node (pp) at ($(0:4.65)+(90:4.25)$) {\footnotesize $\Phi_{+}$};
\coordinate[label=below left: {\small $0$}] (x0) at ($(0:0)$);
\fill (x0) circle (1.5pt);
\coordinate[label=below: {\small $3$}] (x3) at ($(0:3)$);
\fill (x3) circle (1.5pt);
\coordinate[label=below: {\small $6$}] (x6) at ($(0:6)$);
\fill (x6) circle (1.5pt);
\coordinate[label=left: {\small $2$}] (y2) at ($(90:2)$);
\fill (y2) circle (1.5pt);
\coordinate[label=left: {\small $3$}] (y3) at ($(90:3)$);
\fill (y3) circle (1.5pt);
\coordinate[label=left: {\small $5$}] (y5) at ($(90:5)$);
\fill (y5) circle (1.5pt);
\draw[dotted] (x6) --++ ($(90:3)$) --++ ($(180:6)$);
\draw[dotted] ($(0:6)+(90:2)$) --++ ($(180:6)$);
\end{tikzpicture}
&\hspace*{2em}
\begin{tikzpicture}[domain=0:6.1,samples=50,scale=.75]
\draw[->] ($(180:1)$) --++ ($(0:8)$) node[right] {\small $z$}; 
\draw[->] ($(270:.5)$) --++ ($(90:6)$) node {};
\draw plot (\x,{5*(5-sqrt(25-2*(5-sqrt(25-4*\x) ) ))}) node[right] {};
\node (pm2) at ($(0:4)+(90:2.75)$) {\footnotesize $\Phi_{-}^{(2)}$};
\draw[dashed] plot (\x,{\x*(10/25)}) node[right] {$\frac{z}{25}$};
\coordinate[label=below left: {\small $0$}] (x0) at ($(0:0)$);
\fill (x0) circle (1.5pt);
\coordinate[label=below: {\small $6$}] (x6) at ($(0:6)$);
\fill (x6) circle (1.5pt);
\draw[dotted] (x6) --++ ($(90:4.38)$);
\coordinate[label=left: {\small $0.438$}] (y6) at ($(90:4.38)$);
\fill (y6) circle (1.5pt);
\draw[dotted] (y6) --++ ($(0:6)$);
\end{tikzpicture}
\end{tabular}
\caption{Inverse functions $\Phi_{-}$, $\Phi_{+}$, and $\Phi_{-}^{(2)}$. Note the rescaled $y$-axis in the latter.}
\label{F:inv_fcts_plot}
\end{figure}

\begin{proposition}\label{P:Phi_props}
For any $0\leq z\leq 6$,

\medskip

\begin{enumerate}[leftmargin=2em,label={\rm (\roman*)},itemsep=0.5em]
\item\label{P:Phi_props_1} $0\leq\Phi_{-}(z)\leq 2$, $3\leq \Phi_{+}(z)\leq 5$ and
\[
\Phi_{-}(0)=0,\quad\Phi_{-}(6)=2,\quad\Phi_{+}(0)=5\quad\text{and}\quad\Phi_{+}(6)=3.
\]
\item\label{P:Phi_props_2}
$\displaystyle \Phi_{+}(z)-\Phi_{-}(z)=\sqrt{25-4z}$ and in particular $\Phi_{-}(z)< \Phi_{+}(z)$.
\end{enumerate}
\end{proposition}

\begin{lemma}\label{L:equidistant}
For any $0\leq x\leq y\leq 6$, 
\[
\Phi_{-}(y)-\Phi_{-}(x)=\Phi_{+}(x)-\Phi_{+}(y)=\frac{1}{2}(\sqrt{25-4x}-\sqrt{25-4y}).
\]
\end{lemma}

Strict convexity will play a fundamental role in proving the key lemmas in Section~\ref{S:small_gaps} and specially that of the composition $\Phi_{-}^{(2)}$. The next lemma is a consequence of this property.

\begin{lemma}\label{L:growing_sizes}
The functions $\Phi_{-}$ and $\Phi_{-}^{(2)}$ are strictly convex on $[0,6]$. In addition, for any $0\leq w\leq x\leq y\leq z\leq 6$,
\begin{enumerate}[leftmargin=2em,label={\rm (\roman*)},itemsep=0.5em]
\item\label{L:g_s-} If $x-w\leq z-y$, then $\displaystyle \Phi_{-}(x)-\Phi_{-}(w)\leq \Phi_{-}(z)-\Phi_{-}(y)$.
\item\label{L:g_s-2} $\displaystyle \Phi_{-}(x)-\Phi_{-}(w)\leq x-w$.
\end{enumerate}
\end{lemma}

\begin{remark}
All items hold with strict inequality when $0\leq w< x< y< z\leq 6$.
\end{remark}

\begin{proof}
To prove (i), the mean value theorem yields
\begin{equation*}
 \Phi_{-}(x)-\Phi_{-}(w)=\Phi_{-}'(\xi_{wx})(x-w)\qquad w\leq\xi_{wx}\leq x
\end{equation*}
and the same for $y,z$ instead of $w,x$. 
Since $\xi_{wx}\leq x\leq y\leq \xi_{yz}$ and $\Phi_{-}'(\xi)=(25-4\xi)^{-1/2}$, we have $\Phi_{-}'(\xi_{wx})\leq \Phi_{-}'(\xi_{yz})$, hence
\[
 \Phi_{-}(x)-\Phi_{-}(w)=\Phi_{-}'(\xi_{wx})(x-w)\leq \Phi_{-}'(\xi_{yz})(z-y)=\Phi_{-}(z)-\Phi_{-}(y).
\]
Part (ii) follows by symmetry since $\Phi_{+}'(\xi)=-\Phi_{-}'(\xi)$ and (iii) because $\Phi_{-}'([0,6])\subset (0,1)$.
\end{proof}


\subsection{Small gaps at finite level}
This section analyzes the Dirichlet and Neumann spectrum of the finite graph Laplacian $\Delta_m$. The type of computations in both cases are of the same nature, however the Dirichlet case is strikingly much less straightforward and requires a delicate analysis of the inverse function $\Phi_{-}$. 

\medskip

In terms of general notation, $\Phi_{\pm}^{(m)}$ denotes the $m$-th concatenation of $\Phi_{\pm}$ and for completeness $\Phi_{\pm}^{(0)}:={\rm id}$.

\subsubsection{Dirichlet spectrum}
Following~\cite[Theorem 3.1]{FS92} and~\cite[Proposition 5.1]{Kig98}, the Dirichlet spectrum of $\Delta_m$ can be described recursively as
\begin{align}\label{E:def_AmD}
&A_0^D:=\emptyset\qquad A_1^D:=\{2,5\}\qquad A_2^D:=\Phi_{\pm}(A_1)\cup\{5,6\}\notag\\
&A_{m}^D:=\Phi_{\pm}(A_{m-1}^D{\setminus}\{6\})\cup\{3,5,6\},
\quad m\geq 3.
\end{align}


To begin with we determine the spectral gap of $\Delta_{m}$, that is the difference between the first two eigenvalues in $A_m^D$. 

\begin{lemma}\label{L:gap1_AmD}
For any $m\geq 1$, the Dirichlet spectral gap of $\Delta_m$ equals $\Phi_{-}^{(m-1)}(5)-\Phi_{-}^{m-1}(2)$. 
\end{lemma}

\begin{proof}
We prove by induction that
\begin{equation*}\label{E:gap1_AmD_01}
\min(A_m^D)=\Phi_{-}^{(m-1)}(2)\qquad\text{ and }\qquad\min(A_m^D{\setminus}\{\Phi_{-}^{(m-1)}(2)\})=\Phi_{-}^{(m-1)}(5).
\end{equation*}
Note that, due to the construction of $A_m^D$, the first three cases need to be considered separately. To ease the notation we write $A_m$ without superscript in what follows.
\begin{itemize}[leftmargin=1em]
\item Case $m=1$. Clear by direct inspection since $\Phi_{-}^{(0)}(2)=2$ and $\Phi_{-}^{(0)}(5)=5$.
\item Case $m=2$. Again by direct inspection, Proposition~\ref{P:Phi_props} and Lemma~\ref{L:growing_sizes} allow to describe elements of $A_2$ in increasing order as
\begin{equation*}
A_2=\{\Phi_{-}(2),\Phi_{-}(5),\Phi_{+}(5),\Phi_{+}(2),5,6\}.
\end{equation*} 
\item Case $m=3$. By construction, $A_3=\Phi_{-}(A_2{\setminus}\{6\})\cup\Phi_{+}(A_1{\setminus}\{6\})\cup\{3,5,6\}$. Proposition~\ref{P:Phi_props} and Lemma~\ref{L:growing_sizes} now yield
\begin{equation*}\label{E:min_AmD_03}
\min (A_3)=\min(\Phi_{-}(A_2))=\Phi_{-}(\min A_2)=\Phi_{-}\Phi_{-}(2).
\end{equation*} 
Analogously the second smallest element is $\Phi_{-}\big(\min (A_2{\setminus}\{\Phi_{-}(2)\})\big)=\Phi_{-}^{(2)}(5)$.
\item General case. Using the hypothesis of induction, the same arguments as before apply so that
\begin{equation}\label{E:gap1_AmD_04}
\min (A_{m+1})=\min(\Phi_{-}(A_{m}))=\Phi_{-}(\min A_{m})=\Phi_{-}\Phi_{-}^{(m-1)}(2)=\Phi_{-}^{(m)}(2)
\end{equation} 
and for the second smallest element $\Phi_{-}\big(\min (A_m{\setminus}\{\Phi_{-}^{(m-1)}(2)\})\big)=\Phi_{-}^{(m)}(5)$.
\end{itemize}
\end{proof}

An immediate consequence of~\eqref{E:gap1_AmD_04} and the fact that $\Phi_{+}$ is monotone decreasing is that $\Phi_{+}\big(\min(A^D_m{\setminus}\{5,6\})\big)=\max\big(\Phi_{+}(A^D_m{\setminus}\{5,6\})\big)$. This provides the largest $2$-series eigenvalue in $A_m^D$, which will be relevant in the proof of the main Lemma~\ref{L:induction_step_D}.
\begin{corollary}\label{C:max_2series_AmD}
For any $m\geq 2$, $\max(A_{m}^D{\setminus}\{5,6\})=\Phi_{+}\Phi_{-}^{(m-2)}(2)$.
\end{corollary}

\begin{remark}\label{R::max_2series_AmD}
The eigenvalue $\Phi_{+}\Phi_{-}^{(m-2)}(2)$ is the closest to the eigenvalue $5$ in $A_m^D$. 
\end{remark}

We proceed to prove two technical but fundamental observations in the quest of finding the exact minimal gap in a subset of the Dirichlet spectrum of $\Delta_m$ relevant for Theorem~\ref{T:distance_lower_bound}. The first one concerns the ratio between the second smallest eigenvalue at level $m$ and the spectral gap at level $m+1$: this ratio becomes larger as the level increases. 

\begin{lemma}\label{L:key2}
For any $m\geq 1$,
\begin{equation}\label{E:key2}
\frac{\Phi_{-}^{(m-1)}(5)}{\Phi_{-}^{(m)}(5)-\Phi_{-}^{(m)}(2)}<\frac{\Phi_{-}^{(m)}(5)}{\Phi_{-}^{(m+1)}(5)-\Phi_{-}^{(m+1)}(2)}.
\end{equation}
\end{lemma}
\begin{proof}
Let $m\geq 1$. Multiplying and dividing by $\Phi_{-}^{(m)}(5)$ the left hand side of~\eqref{E:key2}, the mean value theorem and the strict convexity of $\Phi_{-}$ yield
\begin{align*}
\frac{\Phi_{-}^{(m-1)}(5)}{\Phi_{-}^{(m)}(5)-\Phi_{-}^{(m)}(2)}
<\frac{\Phi_{-}^{(m-1)}(5)}{\Phi_{-}^{(m)}(5)}\Phi_{-}'(\Phi_{-}^{(m)}(5))\frac{\Phi_{-}^{(m)}(5)}{\Phi_{-}^{(m+1)}(5)-\Phi_{-}^{(m+1)}(2)}.
\end{align*}
It remains to show that 
\begin{equation}\label{E:key2_01}
\frac{\Phi_{-}^{(m-1)}(5)}{\Phi_{-}^{(m)}(5)}\Phi_{-}'(\Phi_{-}^{(m)}(5))<1.
\end{equation}
Using the relation~\eqref{E:ev_level_rel} with $\lambda_{m-1}=\Phi_{-}^{(m-1)}(5)$ and $\lambda_{m}=\Phi_{-}^{(m)}(5)$, and the explicit expression of the derivative $\Phi_{-}'$, condition~\eqref{E:key2_01} is equivalent to
\begin{equation*}
\frac{5-\Phi_{-}^{(m)}(5)}{\sqrt{25-4\Phi_{-}^{(m)}(5)}}<1\quad\Leftrightarrow\quad (5-\Phi_{-}^{(m)}(5))^2<25-4\Phi_{-}^{(m)}(5)\quad\Leftrightarrow\quad \Phi_{-}^{(m)}(5)<6
\end{equation*}
which is clearly true for any $m\geq 1$.
\end{proof}

The second observation relies on the former and pertains the fact that the ratio between the smallest eigenvalue at level $m$ and the spectral gap at level $m+2$ becomes larger as the level increases. The reason for comparing between ``two-level steps'' instead of consecutive levels is not obvious; this insight was gained after a thorough numerical analysis, which also hints that the ratio grows \emph{slower} as the level increases.

\begin{lemma}\label{L:key1}
For any $m\geq 1$,
\begin{equation}\label{E:key1}
\frac{\Phi_{-}^{(m)}(2)}{\Phi_{-}^{(m+2)}(5)-\Phi_{-}^{(m+2)}(2)}>\frac{\Phi_{-}^{(m-1)}(2)}{\Phi_{-}^{(m+1)}(5)-\Phi_{-}^{(m+1)}(2)}.
\end{equation}
\end{lemma}
\begin{proof}
Let $m\geq 1$. Adding and subtracting $\Phi_{-}^{(m+1)}(5)$ in the numerator of the left hand side of~\eqref{E:key1}, applying the mean value theorem and since $\Phi_{-}^{(m)}(5)>\Phi_{-}^{(m+1)}(5)$ we obtain
\begin{align}
\frac{\Phi_{-}^{(m)}(2)}{\Phi_{-}^{(m+2)}(5)-\Phi_{-}^{(m+2)}(2)}&=\frac{\Phi_{-}^{(m)}(2)-\Phi_{-}^{(m+1)}(5)}{\Phi_{-}^{(m+2)}(5)-\Phi_{-}^{(m+2)}(2)}+\frac{\Phi_{-}^{(m+1)}(5)}{\Phi_{-}^{(m+2)}(5)-\Phi_{-}^{(m+2)}(2)}\notag\\
&>\frac{\Phi_{-}^{(m-1)}(2)-\Phi_{-}^{(m)}(5)}{\Phi_{-}^{(m+1)}(5)-\Phi_{-}^{(m+1)}(2)}+\frac{\Phi_{-}^{(m+1)}(5)}{\Phi_{-}^{(m+2)}(5)-\Phi_{-}^{(m+2)}(2)},\label{E:key1_01}
\end{align}
where the inequality follows from strict convexity
. Reordering terms we may write~\eqref{E:key1_01} as
\begin{equation*}
\frac{\Phi_{-}^{(m-1)}(2)}{\Phi_{-}^{(m+1)}(5)-\Phi_{-}^{(m+1)}(2)}+\bigg(\frac{\Phi_{-}^{(m+1)}(5)}{\Phi_{-}^{(m+2)}(5)-\Phi_{-}^{(m+2)}(2)}-\frac{\Phi_{-}^{(m)}(5)}{\Phi_{-}^{(m+1)}(5)-\Phi_{-}^{(m+1)}(2)}\bigg).
\end{equation*}
By virtue of Lemma~\ref{L:key2}, the quantity in brackets is strictly positive, hence~\eqref{E:key1_01} implies
\[
\frac{\Phi_{-}^{(m)}(2)}{\Phi_{-}^{(m+2)}(5)-\Phi_{-}^{(m+2)}(2)}>\frac{\Phi_{-}^{(m-1)}(2)}{\Phi_{-}^{(m+1)}(5)-\Phi_{-}^{(m+1)}(2)}
\]
as we wanted to prove.
\end{proof}

Finally we are in the position to prove the main lemma that is used in the proof of Theorem~\ref{T:distance_lower_bound}. It determines the minimal spacing between any consecutive eigenvalues in level $m+k$ with generation of fixation $\ell \leq m+1$. The reason why it is enough to focus on those is that, if an eigenvalue $\lambda_{m+k}$ has a higher level of fixation, it is always possible to find a suitable level $m'>m$, where its successor $\lambda_{m'+k-1}$ will have generation of fixation $\ell= m'+1$. 

\medskip

The complete argument proceeds by a double induction on $m$ and $k$ starting at level $m=3$, $k=2$. Starting at $k=2$ turns to be key and the proof of this initial case is rather non-trivial, in particular in view of the following observation that is discussed at the end of this section.
\begin{remark}\label{R:lowest_Dm}
Taking on account all Dirichlet eigenvalues at level $m$, i.e. $k=0$, yields
\[
\min\{|\lambda-\lambda'|\colon \lambda,\lambda'\in A_m^D, \lambda\neq \lambda'\}=\Phi_{-}^{(m-1)}(2),
\]
which is \textit{strictly smaller} than~\eqref{E:induction_step_D}. One can also see that this minimal spacing occurs between the eigenvalues $\lambda=5$ and $\lambda'=\Phi_{+}\Phi_{-}^{(m-2)}(2)$. In the case $k=1$ the corresponding minimum is again \textit{strictly smaller} than~\eqref{E:induction_step_D}; there is strong numerical evidence that it corresponds to the spacing between $\lambda=\Phi_{-}(5)$ and $\lambda'=\Phi_{-}\big(\Phi_{+}\Phi_{-}^{(m-2)}(2)\big)$.
\end{remark}

\begin{lemma}\label{L:induction_step_D}
For any $m\geq 3$,
\begin{equation}\label{E:induction_step_D}
\displaystyle \min\{|\lambda'-\lambda|\colon~\lambda,\lambda'\in\Phi_{-}^{(2)}\big(A_m^D{\setminus}\{6\}\big),\;\lambda\neq\lambda'\}=\Phi_{-}^{(m+1)}(5)-\Phi_{-}^{(m+1)}(2).
\end{equation}
The minimum is attained for $\lambda=\Phi_{-}^{(m+1)}(2)$ and $\lambda'=\Phi_{-}^{(m+1)}(5)$.
\end{lemma}
\begin{proof}
To ease the notation, we write $A_m$ without superscript. 
\begin{itemize}[leftmargin=1em]
\item Case $m=3$.  We start by describing explicitly the elements of the set $A_3{\setminus}\{6\}$ in increasing order
{\small
\[
\{\Phi_{-}^{(2)}(2),\Phi_{-}^{(2)}(5),\Phi_{-}\Phi_{+}(5),\Phi_{-}\Phi_{+}(2),\Phi_{-}(5),3,\Phi_{+}(5),\Phi_{+}^{(2)}(2),\Phi_{+}^{(2)}(5),\Phi_{+}\Phi_{-}(5),\Phi_{+}\Phi_{-}(2),5\}.
\]
}
The size of the spacing between any two consecutive eigenvalues $\Phi_{-}^{(2)}(\lambda)$, $\Phi_{-}^{(2)}(\lambda')$ with $\lambda,\lambda'\in A_3{\setminus}\{6\}$ can be explicitly computed and compared to the size of the spectral gap at level $5$, that is $\Phi_{-}^{(4)}(5)-\Phi_{-}^{(4)}(2)$, see Lemma~\ref{L:gap1_AmD}. Table~\ref{F:induction_step_D_table} shows the \emph{explicit difference} (computed here using Phyton) between the size of each spacing and the spectral (1st) gap.
\begin{table}[H]
\begin{tabular}{l|ccccccc}
gap & 1st-1st & 2nd-1st & 3rd-1st & 4th-1st & 5th-1st & 6th-1st \\ \hline
difference & 0 & 0.0164 & 0.0061 & 7.9131 & 0.0758 & 0.0303\\ \hline\\ \hline
gap & 1st-1st & 7th-1st & 8th-1st & 9th-1st & 10th-1st & 11th-1st \\ \hline
difference & 0 & 0.0039 & 0.0149 & 0.0395 & 0.0108 & 0.0005\\
\end{tabular}
\caption{Values approximated to four digits.}
\label{F:induction_step_D_table}
\end{table}
\item $(m)\Rightarrow(m+1)$. Let $\lambda,\lambda'\in\Phi_{-}^{(2)} (A_{m+1}{\setminus}\{6\})$ with $\lambda<\lambda'$ and $\lambda=\Phi_{-}(\mu)$, $\lambda'=\Phi_{-}(\mu')$ for some $\mu<\mu'$ in $\Phi_{-} (A_{m+1}{\setminus}\{6\})$. Since by construction $A_{m+1}{\setminus}\{6\}=\Phi_{\pm}(A_{m})\cup\{3,5\}$, we have that
\[
\mu,\mu'\in \Phi_{-}^{(2)} (A_{m}{\setminus}\{6\})\cup\{\Phi_{-}(3)\}\cup\Phi_{-}\Phi_{+}(A_{m}{\setminus}\{6\})\cup\{\Phi_{-}(5)\}.
\]
We analyze the possible situations for the gap $\lambda'-\lambda$ based on the subsets $\mu$ and $\mu'$ belong. 
\begin{enumerate}[leftmargin=1.5em,label=(\alph*)]
\item If $\mu,\mu'\in \Phi_{-}^{(2)} (A_{m}{\setminus}\{6\})$, the strict convexity of $\Phi_{-}$ and the induction hypothesis yield
\begin{align*}
\lambda'-\lambda&=\Phi_{-}(\mu')-\Phi_{-}(\mu)\geq \Phi_{-}'(\xi_{\Phi_{-}^{(m+1)}(2),\Phi_{-}^{(m+1)}(5)})(\mu'-\mu)\\
&\geq \Phi_{-}'(\xi_{\Phi_{-}^{(m+1)}(2),\Phi_{-}^{(m+1)}(5)})\big(\Phi_{-}^{(m+1)}(5)-\Phi_{-}^{(m+1)}(2)\big)=\Phi_{-}^{(m+2)}(5)-\Phi_{-}^{(m+2)}(2).
\end{align*}
\item If $\mu,\mu'\in \Phi_{-}\Phi_{+}(A_{m}{\setminus}\{6\})$, then $\mu=\Phi_{-}\Phi_{+}(\nu)$ and $\mu'=\Phi_{-}\Phi_{+}(\nu')$ for some $\nu'<\nu$ belonging to $A_{m}{\setminus}\{6\}$. In particular, c.f.~Lemma~\ref{L:growing_sizes}, $\Phi_{+}(\nu)>\Phi_{-}(\nu)$ which together with the strict convexity of $\Phi_{-}^{(2)}$ and Lemma~\ref{L:equidistant} yields
\begin{align*}
\lambda'-\lambda&=\Phi_{-}^{(2)}\Phi_{+}(\nu')-\Phi_{-}^{(2)}\Phi_{+}(\nu)=\Phi_{-}^{(2)'}(\xi_{\Phi_{+}(\nu),\Phi_{+}(\nu)})(\Phi_{+}(\nu')-\Phi_{+}(\nu))\\
&>\Phi_{-}^{(2)'}(\xi_{\Phi_{-}(\nu'),\Phi_{-}(\nu')})(\Phi_{+}(\nu')-\Phi_{+}(\nu))=\Phi_{-}^{(2)'}(\xi_{\Phi_{-}(\nu'),\Phi_{-}(\nu')})(\Phi_{-}(\nu')-\Phi_{-}(\nu))\\
&=\Phi_{-}^{(3)}(\nu')-\Phi_{-}^{(3)}(\nu).
\end{align*}
Since $\Phi_{-}^{(3)}(\nu'),\Phi_{-}^{(3)}(\nu)\in \Phi_{-}^{(2)} (A_{m}{\setminus}\{6\})$, the previous case (a) applies and hence $\lambda'-\lambda>\Phi_{-}^{(m+2)}(5)-\Phi_{-}^{(m+2)}(2)$.
\item If $\mu=\Phi_{-}(3)$ or $\mu'=\Phi_{-}(3)$, we show that $\mu'=\Phi_{-}\Phi_{+}(5)$ or $\mu=\Phi_{-}^{(2)}(5)$ respectively, whence $\lambda,\lambda'\in\Phi_{-}^{(2)}(A_3{\setminus}\{6\})$ and the claim follows from the induction start $m=3$. Indeed, since $3=\Phi_{+}(6)$, its closest eigenvalues in $A_{m+1}{\setminus}\{6\}$ are $\Phi_{+}(5)$ and $\Phi_{-}(5)$.
\item If $\mu'=\Phi_{-}(5)$, Corollary~\ref{C:max_2series_AmD} implies $\mu=\Phi_{-}\Phi_{+}\Phi_{-}^{(m-1)}(2)$ so that $\lambda=\Phi_{-}^{(2)}\Phi_{+}\Phi_{-}^{(m-1)}(2)$ and $\lambda'=\Phi_{-}^{(2)}(5)$. Using the strict convexity of $\Phi_{-}^{(2)}$ and the explicit expression of the derivative $\Phi_{-}^{(2)'}(z)=\Phi_{-}'(\Phi_{-}(z))\Phi_{-}'(z)$ we obtain
\begin{align*}
\lambda'-\lambda&=\Phi_{-}^{(2)'}(\xi_{\Phi_{+}\Phi_{-}^{(m-1)}(2),5})\big(5-\Phi_{+}\Phi_{-}^{(m-1)}(2)\big)\\
&=\Phi_{-}^{(2)'}(\xi_{\Phi_{+}\Phi_{-}^{(m-1)}(2),5})\Phi_{-}^{(m)}(2)\\
&>\Phi_{-}^{(2)'}(\Phi_{+}\Phi_{-}^{(m-1)}(2))\Phi_{-}^{(m)}(2)\\
&=\Phi_{-}'(\Phi_{-}\Phi_{+}\Phi_{-}^{(m-1)}(2))\Phi_{-}'(\Phi_{+}\Phi_{-}^{(m-1)}(2)))\Phi_{-}^{(m)}(2)\\
&=\frac{\Phi_{-}^{(m)}(2)}{\sqrt{25-4\Phi_{-}\Phi_{+}\Phi_{-}^{(m-1)}(2)}\sqrt{25-4\Phi_{+}\Phi_{-}^{(m-1)}(2))}}.
\end{align*}
Multiplying and dividing the latter by $\Phi_{-}^{(m+2)}(5)-\Phi_{-}^{(m+2)}(2)$ it follows that
{\small
\begin{equation}\label{E:induction_step_D_01}
\lambda'-\lambda>
\frac{\Phi_{-}^{(m)}(2)/\big(\Phi_{-}^{(m+2)}(5)-\Phi_{-}^{(m+2)}(2)\big)}{\sqrt{25-4\Phi_{-}\Phi_{+}\Phi_{-}^{(m-1)}(2)}\sqrt{25-4\Phi_{+}\Phi_{-}^{(m-1)}(2))}}\big(\Phi_{-}^{(m+2)}(5)-\Phi_{-}^{(m+2)}(2)\big).
\end{equation}
}
Note now that $m> 3$, hence $\Phi_{-}\Phi_{+}\Phi_{-}^{(m-1)}(2)>\Phi_{-}\Phi_{+}\Phi_{-}^{(2)}(2)$ and $\Phi_{+}\Phi_{-}^{(m-1)}(2)>\Phi_{+}\Phi_{-}^{(2)}(2)$. Together with Lemma~\ref{L:key1},  the right hand side of~\eqref{E:induction_step_D_01} is bounded below by
\[
\frac{\Phi_{-}^{(3)}(2)/\big(\Phi_{-}^{(5)}(5)-\Phi_{-}^{(5)}(2)\big)}{\sqrt{25-4\Phi_{-}\Phi_{+}\Phi_{-}^{(2)}(2)}\sqrt{25-4\Phi_{+}\Phi_{-}^{(2)}(2)}}\big(\Phi_{-}^{(m+2)}(5)-\Phi_{-}^{(m+2)}(2)\big)
\]
and an explicit computation reveals that 
\[
\frac{\Phi_{-}^{(3)}(2)/\big(\Phi_{-}^{(5)}(5)-\Phi_{-}^{(5)}(2)\big)}{\sqrt{25-4\Phi_{-}\Phi_{+}\Phi_{-}^{(2)}(2)}\sqrt{25-4\Phi_{+}\Phi_{-}^{(2)}(2)}}>1
\]
whence $\lambda'-\lambda>\Phi_{-}^{(m+2)}(5)-\Phi_{-}^{(m+2)}(2)$ as we wanted to prove.
\end{enumerate}
\end{itemize}
\end{proof}

\begin{lemma}\label{L:pre-lowest_D}
For any $m\geq 3$ and $k\geq 2$,
\begin{equation}\label{E:pre-lowest_D}
\displaystyle \min\{|\lambda'-\lambda|\colon~\lambda,\lambda'\in\Phi_{-}^{(k)}\big(A_m^D{\setminus}\{6\}\big)\}=\Phi_{-}^{(m+k-1)}(5)-\Phi_{-}^{(m+k-1)}(2).
\end{equation}
The minimum is attained for $\lambda=\Phi_{-}^{(m+k-1)}(2)$ and $\lambda'=\Phi_{-}^{(m+k-1)}(5)$.
\end{lemma}

\begin{proof}
To ease the notation, we write $A_m$ without superscript. Consider first $m=3$ and apply induction on $k$:
\begin{itemize}[leftmargin=1em]
\item Case $k=2$. See Lemma~\ref{L:induction_step_D}. 
\item $(k)\Rightarrow(k+1)$. Let $\lambda<\lambda'$ belong to $\Phi_{-}^{(k+1)}(A_3{\setminus}\{6\})$. Then, $\lambda=\Phi_{-}(\mu)$ and $\lambda'=\Phi_{-}(\mu')$ for some $\mu,\mu'\in \Phi_{-}^{(k)}(A_3{\setminus}\{6\})$ with $\mu<\mu'$ and $\mu\geq \Phi_{-}^{(k+2)}(2)$, $\mu'\geq \Phi_{-}^{(k+2)}(5)$. Applying the mean value theorem, the strict convexity of $\Phi_{-}$ and the hypothesis of induction yield
\begin{align*}
\lambda'-\lambda&=\Phi_{-}(\xi_{\mu,\mu'}))(\mu'-\mu)\\
&\geq \Phi_{-}(\xi_{\Phi_{-}^{(k+2)}(2),\Phi_{-}^{(k+2)}(5)}))(\Phi_{-}^{(k+2)}(5)-\Phi_{-}^{(k+2)}(2))=\Phi_{-}^{(k+3)}(5)-\Phi_{-}^{(k+3)}(2).
\end{align*}
\end{itemize}
Now we perform induction over the parameter $m$: Assuming that the claim is true for $m$ and any $k\geq 2$, we prove that it also holds for (fixed) $m+1$ and any $k\geq 2$.
\begin{itemize}[leftmargin=1em]
\item Case $k=2$. See Lemma~\ref{L:induction_step_D}.
\item $(k)\Rightarrow(k+1)$. Verbatim to the case $m=3$ substituting $2$ by $m$.
\end{itemize}
\end{proof}

As pointed out in Remark~\ref{R:lowest_Dm}, including all Dirichlet eigenvalues of $\Delta_m$ provides a bound that is \emph{strictly lower} than~\eqref{E:induction_step_D}. Indeed,  Corollary~\ref{C:max_2series_AmD}, Proposition~\ref{P:Phi_props}(i) and Lemma~\ref{L:equidistant} yield
\[
5-\Phi_{+}\Phi_{-}^{(m-2)}(2)=\Phi_{+}(5)-\Phi_{+}\Phi_{-}^{(m-2)}(2)=\Phi_{-}^{(m-1)}(2)
\]
which is \emph{strictly smaller} than $\Phi_{-}^{(m-1)}(5)-\Phi_{-}^{(m-1)}(2)$: 
For $m=1$ this is clear ($2<5-2$) and in general, using the fact that $\Phi_{-}^{(m)}(0)=0$ for any $m\geq 1$, strict convexity and induction yield
\[
\frac{\Phi_{-}^{(m)}(2)}{\Phi_{-}^{(m)}(5)-\Phi_{-}^{(m)}(2)}=\frac{\Phi_{-}^{(m)}(2)-\Phi_{-}^{(m)}(0)}{\Phi_{-}^{(m)}(5)-\Phi_{-}^{(m)}(2)}<\frac{\Phi_{-}^{(m-1)}(2)-\Phi_{-}^{(m-1)}(0)}{\Phi_{-}^{(m-1)}(5)-\Phi_{-}^{(m-1)}(2)}<1.
\]

\subsubsection{Neumann spectrum}
From a general point of view, the description of the Neumann spectrum of $\Delta_m$ is similar the Dirichlet one, however the few changes make the question at hand much easier to tackle. In particular, the absence of $2$-series eigenvalues happens to be specially advantageous. We recall e.g. from~\cite[Proposition 5.5]{Kig98} that the Neumann spectrum can be described recursively as
\begin{equation*}
A_0^N:=\{0,6\}\qquad A_1^N:=\{0,3,6\}
\qquad A_{m}^N:=\Phi_{\pm}(A_{m-1}^N{\setminus} \{6\})\cup\{3,6\}
\quad m\geq 2.
\end{equation*}
Again, note that in this case there are no eigenvalues starting from 2, whereas the eigenvalue $5$ does appear in all $A_m^N$ with $m\geq 2$ since $\Phi_{+}(0)=5$. 

\begin{lemma}\label{L:gap1_AmN}
For any $m\geq 1$, the Neumann spectral gap of $\Delta_m$ equals $\Phi_{-}^{(m-1)}(3)$.
\end{lemma}
\begin{proof}
By similar arguments as to those in Lemma~\ref{L:gap1_AmD} one proves by induction that
\begin{equation*}
\min(A_m^N{\setminus}\{0\})=\Phi_{-}^{(m-1)}(3).
\end{equation*}
\end{proof}

Obtaining the minimal spacing between Neumann eigenvalues of $\Delta_m$ is significantly easier than in the Dirichlet case.
\begin{lemma}\label{L:gap_AmN}
For any $m\geq 1$, 
\[
\min\{|\lambda-\lambda'|\colon \lambda,\lambda'\in A^N_m, \lambda\neq \lambda'\}=\Phi_{-}^{(m-1)}(3).
\]
The minimum is attained for $\lambda=0$ and $\lambda'=\Phi_{-}^{(m-1)}(3)$.
\end{lemma}
%
\begin{proof}
To ease the notation, we write $A^N_m$ without superscript. The first two cases are treated separately and induction starts at $m=3$:
\begin{itemize}[leftmargin=1em]
\item Case $m=1$. By direct inspection, since $A_1=\{0,3,6\}$ the minimum is $3$.
\item Case $m=2$. Writing explicitly $A_2=\{0,\Phi_{-}(3),3,\Phi_{+}(3),5,6\}$ and applying Proposition~\ref{P:Phi_props} and Lemma~\ref{L:equidistant} gives the bounds $3-\Phi_{-}(3)>1>\Phi_{-}(3)$ and $5-\Phi_{+}(3)=\Phi_{-}(3)$. In addition, a direct computation shows $\Phi_{+}(3)-3\approx 1.303>\Phi_{-}(3)\approx 0.6972$.
\item Case $m=3$. By construction, the elements of $A_3$ in increasing order are
\[
\Phi_{-}\big(A_2{\setminus}\{6\}\big)\cup\{3\}\cup\Phi_{+}\big(A_2{\setminus}\{6\}\big)\cup\{6\}.
\]
Given a pair of consecutive eigenvalues $(\lambda,\lambda')$ with $\lambda<\lambda'$: If both belong to $\Phi_{-}(A_{2}{\setminus}\{6\})$, then there exist $\mu',\mu\in A_2{\setminus}\{6\}$ with $\mu'-\mu\geq \Phi_{-}(3)$ from the previous case so that Lemma~\ref{L:growing_sizes}\ref{L:g_s-} yields
\[
\lambda'-\lambda=\Phi_{-}(\mu)-\Phi_{-}(\mu')\geq \Phi_{-}\Phi_{-}(3)=\Phi_{-}^{(2)}(3).
\]
The same is valid for $\lambda,\lambda'\in \Phi_{+}(A_{2}{\setminus}\{6\})$ by Lemma~\ref{L:equidistant}. If $\lambda=3=\Phi_{+}(6)$, then $\lambda'=\Phi_{+}(5)$ and Lemma~\ref{L:equidistant} yields $\lambda'-\lambda=(\sqrt{5}-1)/2\approx 0.618 >\Phi_{-}^{(2)}(3)\approx 0.0351$. If $\lambda'=3$, then $\lambda=\Phi_{-}(5)$ and Proposition~\ref{P:Phi_props} yields $\lambda'-\lambda> 1$.
\item Induction. Writing again
$ A_{m+1}=\Phi_{-}(A_{m}{\setminus}\{6\})\cup\{3\}\cup\Phi_{+}(A_{m}{\setminus}\{6\})\cup\{6\}$ the hypothesis of induction yields
\[
\Phi_{-}^{(m-1)}(3)\leq \mu'-\mu\qquad\text{ for any }\mu<\mu'\;\text{in } A_{m}{\setminus}\{6\}.
\]
If the pair of consecutive eigenvalues $(\lambda,\lambda')$ has both $\lambda,\lambda'\in\Phi_{-}(A_m{\setminus}\{6\})$, there are $\mu\geq 0$ and $\mu'\geq\Phi_{-}^{(m-1)}(3)$ such that 
\[
\lambda'-\lambda=\Phi_{-}(\mu)-\Phi_{-}(\mu')\geq \Phi_{-}\Phi_{-}^{(m-1)}(3)=\Phi_{-}^{(m)}(3),
\]
where the inequality follows from Lemma~\ref{L:growing_sizes}. The same applies to $\lambda,\lambda'\in\Phi_{+}(A_m{\setminus}\{6\})$ by virtue of Lemma~\ref{L:equidistant}. Further, since $3=\Phi_{+}(6)$, its closest eigenvalues in $A_{m+1}$ are $\Phi_{-}(5)$ and $\Phi_{+}(5)$. These gaps are included in the case $m=3$, and bounded below by $\Phi_{-}^{(m)}(3)$ by the strict convexity of $\Phi_{-}$. Finally, we also note $d(\Phi_{+}(A_{m}{\setminus}\{6\}),\{6\})=6-5=1>\Phi_{-}^{(m)}(3)$, hence the proof is complete. 
\end{itemize}
\end{proof}

\section{Conclusion and final remarks}\label{S:end}
Laplace operators on fractals and their spectrum are a recurrent object of study in the physics literature~\cite{RT82,Mal95,ADT09,DGV12,BRBC20}. The present paper investigates the size of the smallest possible gap between any two consecutive Dirichlet or Neumann eigenvalues of the Laplacian on the Sierpinski gasket, which turns out to coincide with the corresponding spectral gap, that is the distance between the first two eigenvalues; see Theorem~\ref{T:distance_lower_bound}. The proof relies on the properties of the dynamical system that describes the spectrum and it opens the possibility that a similar result is true for a class of fractals whose spectrum enjoys similar properties. One possible starting point would be to consider fractals whose Laplacian admits spectral decimation, and find general conditions for the associated inverse functions that would produce the desired outcome.

\bigskip

\subsection*{Acknowledgments}
The author thanks greatly A.Teplyaev for very inspiring discussions.

\bibliographystyle{amsplain}
\bibliography{Spectral_gap_SG}

\providecommand{\bysame}{\leavevmode\hbox to3em{\hrulefill}\thinspace}
\providecommand{\MR}{\relax\ifhmode\unskip\space\fi MR }
\providecommand{\MRhref}[2]{%
  \href{http://www.ams.org/mathscinet-getitem?mr=#1}{#2}
}
\providecommand{\href}[2]{#2}
\begin{thebibliography}{10}

\bibitem{ADT09}
E.~Akkermans, G.~V. Dunne, and A.~Teplyaev, \emph{Physical consequences of
  complex dimensions of fractals}, EPL (Europhysics Letters) \textbf{88}
  (2009), no.~4, 40007.

\bibitem{BBRR17}
V.~Blomer, J.~Bourgain, M.~Radziwi\l, and Z.~Rudnick, \emph{Small gaps in the
  spectrum of the rectangular billiard}, Ann. Sci. \'{E}c. Norm. Sup\'{e}r. (4)
  \textbf{50} (2017), no.~5, 1283--1300.

\bibitem{BRBC20}
A.~Boretti, L.~Rosa, J.~Blackledg, and S.~Castelletto, \emph{A preliminary
  study of a graphene fractal {S}ierpinski antenna}, {IOP} Conference Series:
  Materials Science and Engineering \textbf{840} (2020), 012003.

\bibitem{DSV99}
K.~Dalrymple, R.~S. Strichartz, and J.~P. Vinson, \emph{Fractal differential
  equations on the {S}ierpinski gasket}, J. Fourier Anal. Appl. \textbf{5}
  (1999), no.~2-3, 203--284.

\bibitem{DGV12}
G.~Derfel, P.~J. Grabner, and F.~Vogl, \emph{Laplace operators on fractals and
  related functional equations}, J. Phys. A \textbf{45} (2012), no.~46, 463001,
  34.

\bibitem{DJ16}
D.~Dolgopyat and D.~Jakobson, \emph{On small gaps in the length spectrum}, J.
  Mod. Dyn. \textbf{10} (2016), 339--352.

\bibitem{FS92}
M.~Fukushima and T.~Shima, \emph{On a spectral analysis for the
  {S}ierpi\'{n}ski gasket}, Potential Anal. \textbf{1} (1992), no.~1, 1--35.

\bibitem{GRS01}
M.~Gibbons, A.~Raj, and R.~S. Strichartz, \emph{The finite element method on
  the {S}ierpinski gasket}, Constr. Approx. \textbf{17} (2001), no.~4,
  561--588.

\bibitem{HSTZ12}
K.~E. Hare, B.~A. Steinhurst, A.~Teplyaev, and D.~Zhou, \emph{Disconnected
  {J}ulia sets and gaps in the spectrum of {L}aplacians on symmetric finitely
  ramified fractals}, Math. Res. Lett. \textbf{19} (2012), no.~3, 537--553.

\bibitem{Kig98}
J.~Kigami, \emph{Distributions of localized eigenvalues of {L}aplacians on post
  critically finite self-similar sets}, J. Funct. Anal. \textbf{156} (1998),
  no.~1, 170--198.

\bibitem{Kig01}
\bysame, \emph{Analysis on fractals}, Cambridge Tracts in Mathematics, vol.
  143, Cambridge University Press, Cambridge, 2001.

\bibitem{KL93}
J.~Kigami and M.~L. Lapidus, \emph{Weyl's problem for the spectral distribution
  of {L}aplacians on p.c.f. self-similar fractals}, Comm. Math. Phys.
  \textbf{158} (1993), no.~1, 93--125.

\bibitem{Mal95}
L.~Malozemov, \emph{Random walk and chaos of the spectrum. {S}olvable model},
  Chaos Solitons Fractals \textbf{5} (1995), no.~6, 895--907.

\bibitem{Ram84}
R.~Rammal, \emph{Spectrum of harmonic excitations on fractals}, J. Physique
  \textbf{45} (1984), no.~2, 191--206.

\bibitem{RT82}
R.~Rammal and G.~Toulouse, \emph{Random walks on fractal structures and
  percolation clusters}, J. Physique Lett. \textbf{43} (1982), L13--L22.

\bibitem{Str05}
R.~S. Strichartz, \emph{Laplacians on fractals with spectral gaps have nicer
  {F}ourier series}, Math. Res. Lett. \textbf{12} (2005), no.~2-3, 269--274.

\bibitem{Str06}
\bysame, \emph{Differential equations on fractals}, Princeton University Press,
  Princeton, NJ, 2006, A tutorial.

\bibitem{Str12}
\bysame, \emph{Exact spectral asymptotics on the {S}ierpinski gasket}, Proc.
  Amer. Math. Soc. \textbf{140} (2012), no.~5, 1749--1755.

\bibitem{Tep98}
A.~Teplyaev, \emph{Spectral analysis on infinite {S}ierpi\'{n}ski gaskets}, J.
  Funct. Anal. \textbf{159} (1998), no.~2, 537--567.

\bibitem{Zho10}
D.~Zhou, \emph{Criteria for spectral gaps of {L}aplacians on fractals}, J.
  Fourier Anal. Appl. \textbf{16} (2010), no.~1, 76--96.

\end{thebibliography}
\end{document}